\documentclass[11pt]{article}
\usepackage[utf8]{inputenc}
\usepackage{amsfonts}
\usepackage[width=17cm, left=1cm]{geometry}
\usepackage{amssymb}
\usepackage{amsmath}
\usepackage{amsthm}
\usepackage{tikz}
\usetikzlibrary{graphs}
\usepackage[english]{babel}
\newcommand{\tr}{\rm{T\!r}}

\newcommand{\Meas}{\rm{Meas}}
\newtheorem{definition}{Definition}

\newtheorem{theorem}{Theorem}

\begin{document}

\title{Limit spectral measures of matrix distributions of metric triples
}

\author{{A.~M.~Vershik \thanks{St Petersburg Department
of V. A. Steklov Mathematical Institute of Russian Academy
of Sciences, St Petersburg State University, Institute for Information
Transmission Problems of Russian Academy of Sciences (Kharkevich Institute); supported by RSF grant project no. 21-11-00152}},
 {F.~V.~Petrov \thanks{St Petersburg Department
of V. A. Steklov Mathematical Institute of Russian Academy
of Sciences, St Petersburg State University; supported by RSF grant project no. 21-11-00152}}}

\maketitle

\begin{abstract}
A notion of the limit spectral measure of a metric triple (i.e., a metric measure space) is defined.
If the metric is square integrable, then the limit spectral measure
is deterministic and coinsides with the spectrum of the integral
operator in $L^2(\mu)$ with kernel $\rho$. We construct an example
in which there is no deterministic spectral measure.
\end{abstract}

\section{Definitions}

Consider a \emph{metric triple} $\tau=(X,\mu,\rho)$ (here triple is for space-measure-metric), 
where $(X,\mu)$ is standard continuous probability measure  space 
(Lebesgue -- Rokhlin space), and $(X,\rho)$ is a
complete separable metric (Polish) space, and the
measure $\mu$ is a Borel and non-degenerate (=with a full support) measure.
For the metric triple $\tau=(X,\mu,\rho)$
consider the so-called \emph{matrix distribution} \cite{Ve}, i.e. a measure on infinite matrices
of distances $$M_{\tau}=(\rho(x_i,x_j))_{i,j=1}^\infty,$$
where $\{x_i\}_{i=1,\dots}$ is an infinite sequence of independent random variables distributed 
by the measure $\mu$.
This measure has important properties and, 
first of all, it is invariant and ergodic with respect to the infinite symmetric group acting on matrices by simultaneous substitutions of rows and columns.
Before discussing other properties of matrix distributions, let us formulate a theorem on the classification of $mm$-spaces.

The exact terminology has not yet been worked out: Gromov's notion \cite{Gr}  ``$mm$-space'' --- that is, a
space with a measure and a metric, has not been defined quite precisely. 
The notion of an ``admissible triple" $(X,\mu,\rho)$ which we introduced earlier in \cite{VPZ} has
 exactly defined meaning, but the name does not seem expressive enough.
 We propose the term "metric triple"; note that the  word ``metric" applies to 
 both measures and metrics (see also the survey \cite{VVZ}).

An isomorphism between two metric triples $\tau_1(X_1,\mu_1,\rho_1)$ and
 $\tau_2=(X_2,\mu_2,\rho_2)$ is an invertible measure-preserving isometry: 
 $$T:X_1\to X_2, T_*(\mu_1)=\mu_2, \rho_2(Tx,Ty)=\rho_1(x,y).$$
Denote the matrix distributions of these triples by $M_{\tau_i}, i=1,2$.

The measures $\mu_1,\mu_2$, if not stated otherwise, are assumed to be continuous and
non-degenerated (that is, positive on all open sets).

\begin{theorem}[Gromov-- Vershik] Two metric triples
 $\tau_1$ and $\tau_2$ are isomorphic if and only if
 their matrix distributions, as measures on matrices, coincide: $M_{\tau_1}= M_{\tau_2}$.
\end{theorem}

Concerning this theorem, see \cite{Gr,Ve}, and the survey \cite{VVZ}.

Thus, every metric triple by this theorem is uniquely determib  ned by its matrix distribution upto isomorphism.

For each metric triple $\tau=(X,\mu,\rho)$, we may try to associate the integral operator
$I_\tau$ for which the metric is the kernel: $$(I_{\tau}f)(y)=\int \rho(x,y)f(x)d\mu(x).$$

If the metric $\rho$ is square summable with repect to the measure
 $\mu$: $\int_{X \times X}\rho(x,y)^2 d\mu(x)d\mu(y)<\infty,$ then this operator
 is well defined in the space $L^2_{\mu}(X)$, it is self-adjoint and is a Hilbert -- Schmidt operator. 
 Thus, it is compact and has a purely point spectrum, which is a square summable sequence of real
 numbers. But if the metric is only summable, then the
 natural domain of $I_{\tau}$ is rather intricate,  and dealing with spectrum requires  caution.

For example, under the assumption of integrability, in particular in the case of compactness or finiteness of the diameter of the metric space, this operator is known to be defined in the space of measurable bounded functions $L^{\infty}$.
It seems to be difficult to express in terms of metric the possibility of considering
this operator as a (at least unbounded closed) operator in $L^2_{\mu}(X)$ and say something about its spectrum.

But our approach of studying metric triples concerns not this potentially definable operator, but
matrix distributions of the space.

For a matrix $M_n$, an upper minor of order $n$ of a random matrix
$M_\tau$, consider the corresponding spectral measure on the real line,
which we think about as the sum of unit delta-measures in the eigenvalues of $M_n$ 
(multiplicity respected). On teh set
$\Meas$ of measures on the line consider the topology defined by the functionals 
$\mu\to \int fd\mu$, where $f$ is a continuous function on the line
with a bounded support not containing 0.

\begin{definition}
The spectral measure of the metric triple is the limit in the space $\Meas({\Bbb R})$, if it exists, 
of a sequence of measures of the form $d\mu_n(x\cdot a_n)$, where $a_n$ is some normalizing sequence of positive numbers.
\end{definition}

Our goal is to find out the conditions on the metric that guarantee the existence of the spectral measure ,
its computation and, most importantly, to investigate whichinformation about the properties of the triple  does spectral measure
retain. At present these questions have not been studied yet, if we do not have in mind the case of finiteness 
of the integral of the square of the metric. 
As far as we know, even the answer to the question posed earlier by the first author as to whether the
$mm$-space is determined by a spectral measure, is open.
Most likely, the general answer is negative,
but it is not clear whether it happens to be positive in any sufficiently rich class of metrics. It is interesting to relate
metric invariants (e.g., entropy) and spectral invariants of metric triples.

In numerous works the problem similar to the above one has been studied, but not for metrics and $mm$-spaces, but for
symmetric kernels of general form. The main consequence for
our case is that for a metric integrable with the square the spectral measure
(to be normalized by $a_n\equiv n$)
coincides with the purely point spectrum of the operator $I_\tau$ (see above), that is, the sum of the delta-measures in its
eigenvalues, cf. \cite{KG}[Theorem 3.1].

\section{Main question and example}

Our main example shows that there may be no deterministic limit spectral measure.
The authors are not aware of similar examples in the theory of U-statistics, but we also consider them, apparently for the first time, in application to matrix distributions.

The main question is:
 \emph{for which metric triples the limit spectral measure (as a measure on $\Meas$) is
 a non-degenerate measure?}

 We give a simple example of the situation, not calculating the spectral measure completely, but showing only
 nondeterminacy of some functional on the limit measure.

 Namely, we consider the behavior of the trace of the square of the minor of oder $n$
 and show
 that it already has a nondegenerate distribution in the limit.

 This distribution is related to the behavior of $U$-statistics of the form $$\sum_{i<j} h(x_i,x_j),$$
 where $x_i$ are i.i.d., but the first moment of a non-negative random variable
 $h(x_1,x_2)$ is infinite. We have been unable to find such a general statement in the extensive literature on
 $U$-statistics (see \cite{KB,BG} et al.)

 \begin{theorem}
 Consider as a metric triple a real line with the usual Euclidean metric and
 as a measure --- the standard Cauchy distribution $\pi^{-1}\frac{dx}{1+x^2}$.
 Then the normlaized by $n^3$ traces of the squares of the minors of the matrix distribution 
 converge to some non-degenerate $\alpha$-stable distribution with $\alpha=1/2$.
   \end{theorem}

   \begin{proof}
   The trace of the square of our random $n\times n$-matrix $M_n$ equals to the sum of squares of its entries, i.e.,
   $$\frac12 \rm tr M_n^2=\sum_{1\leqslant i<j\leqslant n} |x_i-x_j|^2=n\sum x_i^2-(\sum x_i)^2,$$
   where $(x_1,\ldots,x_n)$ is Cauchy i.i.d. sample.
  The random variable $\frac 1n \sum x_i$ again has the standard Cauchy distribution, thus 
   $(\sum x_i)^2$ is of order $n^2$, in particular,  $n^{-3}(\sum x_i)^2$ converges to 0 in probability. On the other hand, $x_i^2$ 
   is a non-negative random variable with tails of order $${\rm prob} (x_1^2>t)={\rm prob} (|x_1|>\sqrt{t})\sim \frac1{\pi \sqrt{t}}, t\to \infty.$$
   It is known \cite{IL} that in this case $n^{-2}(x_1^2+\ldots+x_n^2)$ 
   converges in distribution to the appropriate $1/2$-stable law. Thus so does $\frac1{2n^3} \tr M_n^2$.
   \end{proof}

It would be interesting to describe the distribution not only of the sum of squares
of eigenvalues, but also of arbitrary powers and the eigenvalues themselves.

\section{Formulation of the general problem and relation to random matrices}

In \cite{Bog}, at the suggestion of the first author, preliminary calculations of spectra of matrix distributions were done
for the most natural triples: $n$-dimensional spheres with Euclidean metric normalized by a Lebesgue measure. 
The results of the calculations indicated a significant difference in the spectral measures already for spheres of small dimensions and not very large orders of minors.
Unfortunately, the authors do not know whether these experiments were continued. Of course, in these cases
due to compactness the limit spectrum is deterministic and should coincide with the spectrum of the integral operator $I$, (which was not clear then) and which apparently has been known for a long time.

The fundamental question is whether there are metric triples (with a non-integrable metric)
whose limit spectrum is similar to the limit spectra of random matrix theory (i.e. semicircular and similar laws.) To clarify this difficult question we should first
compare the theory of matrix distributions with the theory of random matrices.

 Matrix distributions as measures on the space of infinite matrices are $S_{\infty}$-invariant and ergodic with respect to the infinite symmetric group. It makes sense to compare the collection of such measures with a similar 
 collection of $U_{\infty}$-invariant ergodic measures on
 of the space of infinite Hermitian matrices, 
 which correspond to Wishart statistics, cf. \cite{Pic, OV, VPEMJ}.
 The common feature of these and other measures is that they are parameterized by similar sets of frequencies, which differs from measures on matrix spaces with independent elements (GOE, GUE, white noise.)

\end{document}